\documentclass[a4]{article}
\usepackage[cp1251]{inputenc}
\usepackage[russian]{babel}
\usepackage{hyperref}

\usepackage{latexsym}
\usepackage{amsmath}
\usepackage{amsfonts}
\usepackage{amssymb}
\newcommand{\enabstract}[1]{
      \addtocontents{enabstr}{%
        \selectlanguage{english}%
        {\it \@enauthor} {\bf\MakeUppercase{\@entitle.}} #1\par\vskip 2mm%
        \selectlanguage{russian}%
      }
}

\newcommand{\Tr}[1]{{{\mathrm Tr}{#1}}}
\def\F{{\mathbb{F}}}

\newcommand{\Image}[1]{{{\mathrm{Im}}{#1}}}

\newtheorem{theorem}{Теорема}
\newtheorem{lemma}{Лемма}
\newtheorem{note}{Замечание}

\newenvironment{proof}{{$\rhd$}}{$\Box$\\}

\title{О решении одной задачи 
олимпиады 
NSUCrypto'2017
и образах кубических отображений
в двоичных полях
}
\author{
Чиликов А. А.
\footnote{Московский Государственный Технический
Университет им. Н.Э. Баумана,
факультет Информатика и системы управления,
кафедра ИУ-8 Информационная безопасность
}
\footnote{Московский Физико-Технический Институт,
факультет инноваций и высоких технологий,
лаборатория продвинутой комбинаторики и сетевых приложений
}
\footnote{Passware, Research Department,
chilikov@passware.com
}
}

\begin{document}
\maketitle

%


\begin{abstract}
Задача об описании образа нелинейной функции $f(x) = x^3+x$ над произвольным
конечным полем характеристики $2$ была предложена на олимпиаде NSUCrypto
в 2017 году, и обозначена организаторами как нерешенная.
В данной работе предложено полное решение указанной задачи.
\end{abstract}

{\bf \large Ключевые слова: }
конечные поля, нелинейные отображения, криптография, теория кодирования.

{\bf УДК 519.7}

\section{Введение}
\label{Intro}

Международная Олимпиада по криптографии NSUCrypto проводится с 2014 года,
и стала одним из интереснейших мероприятий в мире криптографии. Идея ее
появилась в Новосибирском Государственном Университете,
который и по сей день является главной площадкой и
организатором олимпиады. В состав жюри входят многие авторитетные
специалисты-криптографы из России, 
стран ближнего и дальнего
зарубежья. Хотя олимпиада позиционируется как студенческая, в ней могут
принять участие
все интересующиеся криптографией -- от школьников
до профессионалов-криптографов
(разумеется, призовой зачет в этих категориях раздельный).

Для решения участникам предлагаются задачи различного уровня сложности,
в том числе и открытые математические проблемы. За первые четыре года
участникам было предложено десять задач, отмеченных организаторами как
нерешенные. Одна из них была решена участником в ходе олимпиады
(\cite{NSU:2016}),
по некоторым другим были достигнуты интересные продвижения. Эти
продвижения публикуются и обсуждаются в отчетных публикациях организаторов
(\cite{NSU:2017},\cite{NSU:2016}),
вместе с остальными решениями и итогами олимпиады.

В 2017 году в качестве задачи 1 второго раунда была предложена задача
об описании образа нелинейной функции $f(x) = x^3+x$ над конечным полем
$\F_{2^n}$ (<<The image set>>). Эта задача была отмечена как нерешенная.
В итоговом отчете (\cite{NSU:2017}) организаторы отметили продвижения,
достигнутые двумя
командами. Однако полного решения так и не было предложено. В том же
отчете было отмечено несколько ранее опубликованных результатов,
которые, по мнению организаторов, могут помочь в решении задачи.

Таким образом, насколько нам известно, данная задача считается нерешенной.
В рамках данной работы мы предлагаем полное решение указанной задачи.

Далее в тексте мы будем использовать следующие обозначения:
\begin{itemize}
\item
$\Image{(h)} = \{y \mid \exists x, h(x)=y\}$ -- образ отображения $h$
\item
$\Tr{(x)} = \sum\limits_{i=0}^{n-1} x^{2^i}$ -- след элемента
$x \in \F_{2^n}$
\item
$\gcd(g,h)$ -- наибольший общий делитель $g$ и $h$
\item
$U_d = \{x \mid x^d=1\}$ -- группа корней степени $d$ из $1$
\end{itemize}

Данная работа была проведена с помощью
Российского Научного Фонда Грант N 17-11-01377.

\section{Первичный анализ задачи}

Обозначим искомый образ $\Image{(x^3+x)}$ через $Y^{(n)}$.
Множество $Y^{(n)}$ состоит из тех и только тех точек $\xi$, для которых
уравнение $x^3+x-\xi = 0$ имеет решения в $\F_{2^n}$.

\begin{lemma}
Пусть $h(x) \in \F_{q}[x]$. Тогда уравнение $h(x) =0$
имеет решения в $\F_q$ 
в том и только том случае, когда
$\deg \gcd(h,x^q-x) > 0$.
\end{lemma}
\begin{proof}
Обозначим общий делитель 
$\gcd(h,x^q-x)$ 
через $g(x)$.
Пусть $h(x)$ имеет корень $\alpha \in \F_{2^n}$.
По теореме Безу отсюда следует, что $x-\alpha$
делит 
$h(x)$. 
Кроме того, $x-\alpha$ делит $x^q-x$. Следовательно,
$g(x)$ также делится на $x-\alpha$, а значит,
его степень больше нуля.

Обратно, пусть $\deg g>0$. Следовательно, он отличен от
константы. Поскольку $x^q-x$ разлагается на линейные сомножители, то и $g$
также является произведением некоторых из этих сомножителей (и их количество
больше нуля). Возьмем любой из этих делителей $x-\alpha$. По теореме Безу
получаем $g(\alpha) = 0$. Следовательно
$h(\alpha) = 0$.
\end{proof}

Степень многочлена $f(x)-\xi = x^3+x-\xi$ равна $3$. Поэтому он заведомо
имеет не более трех корней в поле $\F_{2^n}$(с учетом возможной кратности).
Более того, общее число корней в соотвествующем поле разложения
(с учетом кратности) в точности равно трем. Все они лежат в некотором
конечном расширении $\F_{2^n}$. 
Возможные следующие случаи:
\begin{enumerate}
\item
Многочлен $f(x)-\xi$ имеет кратные корни
\item
Все корни многочлена $f(x)-\xi$ различны, ни один из них не лежит в $\F_{2^n}$
\item
Все корни многочлена $f(x)-\xi$ различны, один из них лежит в $\F_{2^n}$
\item
Все корни многочлена $f(x)-\xi$ различны, и все три лежит в $\F_{2^n}$
\end{enumerate}
Вариант, когда все корни различны, и ровно два из них принадлежат $\F_{2^n}$,
невозможен в силу теоремы Виета.

Обозначим теперь через $Y_i^{(n)}$ множество значений $\xi$, для которых
число различных корней $f(x)-\xi$, лежащих в $\F_{2^n}$ в точности равно $i$.

Начнем с рассмотрения первого случая. Из общей теории
(например, \cite{Lidl-Niederreiter:1}) известно, что
кратные корни произвольного многочлена $h$ являются также корнями его
производной $h'$. В нашем случае производная не зависит от $\xi$ и равна
$x^2+1$. У нее имеется единственный корень $x=1$. $f(1) = 0$, поэтому
кратные корни возникают только при $\xi = 0$. Во всех остальных случаях
корни различны.

Вышеприведенное рассуждение доказывает следующий факт:
\begin{lemma}
В ранее введенных обозначениях имеют место следующие разбиения множеств:
\begin{enumerate}
\item
$\F_{2^n} \setminus \{0\} = Y_0^{(n)} \sqcup Y_1^{(n)} \sqcup Y_3^{(n)}$
\item
$Y^{(n)} = \{0\} \sqcup Y_1^{(n)} \sqcup Y_3^{(n)}$
\end{enumerate}
\end{lemma}

Таким образом, задача сводится к описанию множеств $Y_i^{(n)}$ при
$i \in \{ 0, 1, 3 \}$.


Имеет место следующий результат (\cite{Williams},\cite{NSU:2017}):
\begin{theorem}[K.~Williams]
\label{Th:W}
Полином $X^3+aX+b$ над $\F_{2^n}$
\begin{itemize}
\item
раскладывается в произведение трех линейных сомножителей, если
$\Tr{(a^3/b^2)}=\Tr{(1)}$ и корни квадратного уравнения $t^2+bt+a^2=0$
являются кубами в $\F_{2^n}$ (при четном $n$) или же в $\F_{2^{2n}}$
(при нечетном $n$)
\item
раскладывается в произведение одного линейного и одного квадратичного
сомножителей, если
$\Tr{(a^3/b^2)}\neq\Tr{(1)}$
\item
неприводим, если
$\Tr{(a^3/b^2)}=\Tr{(1)}$ и корни квадратного уравнения $t^2+bt+a^2=0$
не являются кубами в $\F_{2^n}$ (при четном $n$) или же в $\F_{2^{2n}}$
(при нечетном $n$)
\end{itemize}
\end{theorem}
В нашем случае требуется исследовать полином $X^3+X+\xi$. Поэтому
$\Tr{(a^3/b^2)} = \Tr{(\xi^{-2})} = \Tr{(\xi^{-1})}$.
Применяя Теорему \ref{Th:W} непосредственно к нашей задаче, получаем
следующее утверждение:
\begin{theorem}
Пусть $\xi \in \F_{2^n} \setminus \{0\}$. Тогда
\begin{itemize}
\item
$\xi \in Y_1^{(n)}$, если $\Tr{(\xi^{-1})}\neq\Tr{(1)}$
\item
$\xi \in Y_3^{(n)}$, если $\Tr{(\xi^{-1})}=\Tr{(1)}$ и
корни квадратного уравнения $t^2+\xi t+1=0$
являются кубами в $\F_{2^n}$ (при четном $n$) или же в $\F_{2^{2n}}$
(при нечетном $n$)
\item
$\xi \in Y_0^{(n)}$, если $\Tr{(\xi^{-1})}=\Tr{(1)}$ и
корни квадратного уравнения $t^2+\xi t+1=0$
не являются кубами в $\F_{2^n}$ (при четном $n$) или же в $\F_{2^{2n}}$
(при нечетном $n$)
\end{itemize}
\end{theorem}
Легко заметить, что полученная классификация допускает разбиение на
два частных случая: $n$ четно и $n$ нечетно. Рассмотрение этих случаев
по отдельности позволяет несколько упростить ситуацию.

\section{Случай четного $n$}

Итак, пусть $n$ четно. Тогда $\Tr{(1)} = 0$. Следовательно, $Y_1^{(n)}$
состоит в точности из тех $\xi$, для которых $\Tr{(\xi^{-1})} \neq 0$,
т.е. $\Tr{(\xi^{-1})} = 1$. Множество точек $\nu = \xi^{-1}$,
таким образом, образует аффинную гиперплоскость $\Tr{(\nu)} = 1$.

Рассмотрим теперь множество $Y_3^{(n)}$. Оно состоит из точек, таких что
$\Tr{(\xi^{-1})} = 0$ и при этом корни уравнения $t^2+\xi t+1=0$ являются
кубами в $\F_{2^n}$. В силу теоремы Виета корни уравнения взаимно обратны,
поэтому если один из них -- куб, то и второй тоже. Обозначим один из
корней через $r^3$. Тогда второй равен $r^{-3}$ и, в силу теоремы Виета,
$\xi = r^{3}+r^{-3}$. 

При любом выборе $r \in \F_{2^n}\setminus\{0\}$
полученное значение $\xi$ будет лежать в $Y_3^{(n)}$. 
Кроме того,
в этом случае заведомо выполнено условие $\Tr{(\xi^{-1})} = 0$.
Также верно и то, что каждое значение
из $Y_3^{(n)}$ будет получено при некотором выборе $r$. Более того,
каждому значению $\xi$ будет соответствовать $6$ возможных значений:
$r$, $rw$, $rw^2$, $r^{-1}$, $r^{-1}w$, $r^{-1}w^2$, где $w$ есть
некоторый фиксированный корень уравнения $w^2+w+1=0$ (т.е. кубический
корень из $1$, не равный $1$). Очевидно, $w \in \F_{2^n}$ при четном $n$.
\begin{note}
В некоторых случаях (при
$\xi\in \{1,w,w^2\}$) некоторые из перечисленных значений могут
совпадать. В этом случае допустимых значений будет в точности $3$, что
легко проверяется прямым подсчетом.
\end{note}
Таким образом, нами построена полная параметризация для множества
$Y_3^{(n)}$ при четном $n$. Его можно несколько упростить, если заметить,
что $r^3$ является корнем степени $\frac{2^n-1}{3}$ из $1$ (и пробегает
все возможные значения из группы $U_{(2^n-1)/3}$).

Зафиксируем итоговый результат данного раздела в виде следующего утверждения:
\begin{theorem}
\label{even:n}
Пусть $n$ четно. Тогда
$$
\Image{(x^3+x)} = Y^{(n)} = \{0\} \sqcup Y_1^{(n)} \sqcup Y_3^{(n)}
$$
где $Y_1^{(n)} = \{\xi \mid \Tr(\xi^{-1}) = 1\}$,
$Y_3^{(n)} = \{r^3+r^{-3} \mid r \in \F_{2^n}, r \neq 0\} = 
\{s+s^{-1} \mid s \in U_{(2^n-1)/3}\}$.
\end{theorem}

\section{Случай нечетного $n$}

Пусть теперь $n$ нечетно. Тогда $\Tr{(1)} = 1$. Следовательно, $Y_1^{(n)}$
состоит в точности из тех $\xi$, для которых $\Tr{(\xi^{-1})} \neq 1$,
т.е. $\Tr{(\xi^{-1})} = 0$. Множество точек $\nu = \xi^{-1}$ таким образом,
образует аффинную гиперплоскость $\Tr{(\nu)} = 0$ с выколотой нулевой точкой.

Рассмотрим теперь множество $Y_3^{(n)}$. Оно состоит из точек, таких что
$\Tr{(\xi^{-1})} = 1$ и при этом корни уравнения $t^2+\xi t+1=0$ являются
кубами в расширенном поле $\F_{2^{2n}}$. Как и ранее, корни уравнения взаимно
обратны. Обозначим один из них через $r^3$.
Тогда второй равен $r^{-3}$ и
$\xi = r^{3}+ r^{-3}$. 


Как и раньше, каждое значение $\xi$ будет получено при некотором выборе $r$.
Однако в этом случае уже нельзя гарантировать, что при любом выборе $r$
результат попадет в основное поле. Очевидно, что вариантов выбора $r$
гораздо больше, чем вариантов выбора $\xi$, и подходящими будут далеко
не все. Для завершения решения задачи нужно научиться отличать <<подходящие>>
значения $r$ от <<неподходящих>>.

Здесь нам поможет следующее вспомогательное утверждение:
\begin{lemma}
\label{L:Frob}
Пусть $f$ -- неприводимый многочлен над конечным полем $\F_q$ и $\vartheta$
-- некоторый его корень в соответствующем алгебраическим расширении.
Тогда $\vartheta^{q}$ также является корнем $f$.
\end{lemma}
\begin{proof}
Указанное утверждение легко проверяется прямым подсчетом. В самом деле,
пусть $f(x) = \sum\limits_{i=0}^{d} a_i x^i$. Тогда
$\sum\limits_{i=0}^{d} a_i \vartheta^i = f(\vartheta) = 0$ и следовательно
$$
0 = {\left(\sum\limits_{i=0}^{d} a_i \vartheta^i \right)}^q =
\sum\limits_{i=0}^{d} a_i (\vartheta^q)^i = f(\vartheta^q)
$$
\end{proof}

Дальнейшее рассуждение несложно. Значение $r$ является <<подходящим>>
тогда и только тогда, когда $r^3$ -- корень уравнения
$t^2+\xi t + 1$. Вторым корнем является, очевидно, $r^{-3}$. 

Если $r^3 \neq 1$, то корни различны. 

С другой стороны, по Лемме \ref{L:Frob}, второй корень равен
$r^{3\cdot2^n}$. Следовательно
$r^{-3} = r^{3\cdot2^n}$ и 
$1 = r^{3 \cdot ( 2^n+1 )}$.
Таким образом, подходящее значение $r$ должно быть
корнем степени $3 \cdot ( 2^n+1 )$ из $1$
(и при этом не лежать в $\F_{2^n}$).

Пусть теперь $\alpha$ -- примитивный элемент поля $\F_{2^{2n}}$.
Тогда $r=\alpha^k$. Условие $1 = r^{3 \cdot ( 2^n+1 )}$ влечет
$3 \cdot ( 2^n+1 ) \cdot k = 0 \bmod 2^{2n}-1$, и следовательно
$3 k = 0 \bmod 2^n-1$. Поскольку $n$ нечетно, $2^n-1$ не делится на $3$.
Значит $2^n-1 \mid k$. Это означает, что $r \in U_{2^n+1}$. Ни
одно из значений $r$, кроме $1$, не лежит в $\F_{2^n}$
(иначе $r^{2^n+1}=1$ и $r^{2^n-1}=1$, что влечет $r^2=1$ и $r=1$).
Таким образом, все такие $r$ являются <<подходящими>>.

Поскольку значение $\xi$ зависит только от $r^3$, можно сразу рассматривать
$s = r^3$. Очевидно, $s^{(2^n+1)/3} = r^{2^n+1} = 1$ (при нечетном $n$
число $(2^n+1)/3$ -- целое).

Таким образом,
$\xi = s+s^{-1}$, где $s$ -- некоторый корень из $1$ степени $(2^n+1)/3$. 
Также ясно, что одному и тому же $\xi$ будут соответствовать два значения ($s$
и $s^{-1}$), а каждому $s$ -- три кубических корня $r$.

Таким образом, нами построена полная параметризация для множества
$Y_3^{(n)}$ при нечетном $n$.

Зафиксируем итоговый результат данного раздела в виде следующего утверждения:
\begin{theorem}
\label{odd:n}
Пусть $n$ нечетно. Тогда
$$
\Image{(x^3+x)} = Y^{(n)} = \{0\} \sqcup Y_1^{(n)} \sqcup Y_3^{(n)}
$$
где $Y_1^{(n)} = \{\xi \mid \Tr(\xi^{-1}) = 0\}$,
$Y_3^{(n)} = \{s+s^{-1} \mid s \in U_{(2^n+1)/3}, s \neq 1\}$.
\end{theorem}

Таким образом, получено описание образа $\Image{(x^3+x)}$ для
произвольного $n$.

\begin{note}
Решение задачи об описании $\Image{(x^3+x)}$ сразу приводит к решению
задачи об описании образа произвольного кубического отображения
$\Image{(ax^3+bx^2+cx+d)}$.
\end{note}
\begin{proof}
При помощи подходящей линейной замены аргумента
$z=ux+v$ можно преобразовать
выражение $ax^3+bx^2+cx+d$ к виду $a(z^3+z)+d'$, либо к виду
$az^3+d'$. Во втором случае описание образа тривиально, поэтому ограничимся
рассмотрением первого. Поскольку преобразование
биективно, $\Image{(ax^3+bx^2+cx+d)} = \Image{(a(z^3+z)+d')}$.
В свою очередь $\Image{(a(z^3+z)+d')}$ получается из образа
$\Image{(x^3+x)}$ аналогичной линейной заменой выходного значения.
\end{proof}

\section{Выводы}

В работе сформулированы и доказаны результаты, позволяющие описать строение
образа кубического отображения $x \to x^3+x$ в произвольном конечном поле
характеристики $2$. Теоремы \ref{even:n} и \ref{odd:n} дают полное описание
указанного образа. При помощи несложных преобразований указанные результаты
обобщаются на случай произвольного кубического отображения.

Полученные результаты представляют интерес в теории конечных полей,
криптографии и алгебраической теории кодирования.



{\center{
{\bf{ABOUT ONE PROBLEM FROM NSUCRYPTO'2017 AND THE IMAGE
OF CUBIC FUNCTION OVER BINARY FIELDS.}}

{\bf{Chilikov A. A.\footnote{
Alexey Chilikov, Ph. D., BMSTU, IU-8;  MIPT; 
Passware, Research Department;Moscow, chilikov@passware.com}}
}}

%
{\bf{Abstract:}} 

The description of the image of cubic function $f(x) = x^3+x$ over
finite field $\F_{2^n}$ was stated as a problem in the NSUCrypto olympiad
in 2017. This problem was marked by organizers as <<unsolved>>. In this
work we propose the full solution of this problem.
%

{\bf{Keywords:}}  
finite fields, non-linear transformations, cryptography, coding theory.



\newcommand{\by}[1]{{\it{#1}~}}
\newcommand{\paper}[1]{{\rm{#1}. }}
\newcommand{\where}[1]{{\rm{#1}. }}

\def \journ{}
\def \jour{}
\def \book{}
\def \yr{}
\def \vol{Vol}
\def \no{№}
\def \pages{p}

\end{document}